\documentclass{amsart}
\usepackage{xcolor}
\usepackage{url}
\usepackage{amsmath}
\usepackage{amssymb}
\usepackage{amsfonts}
\usepackage{amsthm}
\usepackage{mathtools}
\usepackage{enumerate}
\usepackage{verbatim}
\numberwithin{equation}{section}

\newcommand{\OO}{\mathcal{O}}

\newcommand{\QQ}{\mathbb{Q}}
\newcommand{\ZZ}{\mathbb{Z}}
\newcommand{\FF}{\mathbb{F}}

\newcommand{\End}{\text{End}}

\newcommand{\Pp}{\mathfrak{P}}

\DeclareMathOperator{\Cl}{Cl}

\DeclareMathOperator{\Aut}{Aut}

\DeclareMathOperator{\Nrd}{Nrd}
\DeclareMathOperator{\Trd}{Trd}
\DeclareMathOperator{\discrd}{discrd}
\DeclareMathOperator{\disc}{disc}

\DeclareMathOperator{\codiff}{codiff}
\DeclareMathOperator{\Ell}{Ell}
\DeclareMathOperator{\Emb}{Emb}

\newtheorem*{prob}{Problem}

\newtheorem{thm}{Theorem}[section]
\newtheorem{prop}[thm]{Proposition}
\newtheorem{lem}[thm]{Lemma}

\theoremstyle{definition}
\newtheorem{defn}[thm]{Definition}
\newtheorem{rmk}[thm]{Remark}
\newtheorem{alg}[thm]{Algorithm}
\newtheorem{conjecture}[thm]{Conjecture}
\newtheorem{heuristic}[thm]{Heuristic}
\newtheorem*{rmk*}{Remark}
\newtheorem*{ex*}{Example}

\title[Computing endomorphism rings of supersingular elliptic
curves]{Computing endomorphism rings of supersingular elliptic curves
  and connections to pathfinding in isogeny graphs}
\author[Eisentr\"ager, Hallgren, Leonardi, Morrison, and Park]{Kirsten
 Eisentr\"ager, Sean Hallgren, Chris Leonardi, Travis Morrison, and
 Jennifer Park}

\address{Kirsten Eisentr\"ager,
Department of Mathematics,
The Pennsylvania State University,
109 McAllister Building,
University Park, PA 16802, USA}
\email{eisentra@math.psu.edu}

\address{Sean Hallgren,     Department of Computer Science and Engineering,
    The Pennsylvania State University,
    350W Westgate Building,
    University Park, PA  16802, USA}
    \email{hallgren@cse.psu.edu}

 \address{Chris Leonardi, The University of Waterloo, 200 University Ave W, Waterloo, ON N2L 3G1, Canada}
\email{cfoleona@uwaterloo.ca}

 \address{Travis Morrison, Institute for Quantum Computing, The University of Waterloo, 200 University Ave W, Waterloo, ON N2L 3G1, Canada}
  \email{travis.morrison@uwaterloo.ca}

 \address{Jennifer Park, The Ohio State University, 231 W. 18th Ave.,
   MW 512 Columbus, OH 43210, USA}
 \email{park.270@osu.edu}

 \thanks{K.E.\ was partially supported by National Science Foundation
   award CNS-1617802 and a Vannevar Bush Faculty Fellowship from the
   US Department of Defense. S.H.\ was partially supported by National
   Science Foundation awards CCF-1618287 and CNS-1617802, and by a
   Vannevar Bush Faculty Fellowship from the US Department of Defense.
   T.M.\ was supported by funding from the Natural Sciences and
   Engineering Research Council of Canada, the Canada First Research
   Excellence Fund, CryptoWorks21, Public Works and Government
   Services Canada, and the Royal Bank of Canada. J.P.\ was partially
   supported by National Science Foundation award DMS-1902199. This
   work was done in part while K.E.\ and S.H.\ were visiting the Simons Institute for the Theory of Computing.}

\begin{document}

\maketitle

 \begin{abstract}
   Computing endomorphism rings of supersingular elliptic cur\-ves is
   an important problem in computational number theory, and it is also
   closely connected to the security of some of the recently proposed
   isogeny-based cryptosystems.  In this paper we give a new
   algorithm for computing the endomorphism ring of a
   supersingular elliptic curve $E$ defined over $\FF_{p^2}$ that runs, under certain
   heuristics, in time $O((\log p)^2p^{1/2})$.  The algorithm works by
   first finding two cycles of a certain form in the supersingular
   $\ell$-isogeny graph $G(p,\ell)$, generating an order
   $\Lambda \subseteq \End(E)$. Then all maximal orders containing
   $\Lambda$ are computed, extending work of Voight~\cite{V2013}. The
   final step is to determine which of these maximal orders is the
   endomorphism ring. As part of the cycle finding algorithm, we give
   a lower bound on the set of all $j$-invariants $j$ that are
   adjacent to $j^p$ in $G(p,\ell)$, answering a question
   in~\cite{ACLLNSS}.

   We also give a new polynomial-time reduction from computing
   $\End(E)$ to pathfinding in the $\ell$-isogeny graph which is
   simpler in several ways than previous ones.
  We show that this reduction leads to another algorithm for computing
  endomorphism rings which runs in time $\tilde O(p^{1/2})$.  This
  allows us to break the second preimage resistance of a hash function
  in the family constructed in~\cite{CGL2009}.
\end{abstract}

\section{Introduction}
Computing the endomorphism ring of an elliptic curve defined over a
finite field is a fundamental problem in computational arithmetic
geometry. For ordinary elliptic curves the fastest algorithm is due to
Bisson and Sutherland \cite{BS2011} who gave a subexponential time
algorithm to solve this problem.  No subexponential time algorithm is
known for general supersingular elliptic curves.

Computing endomorphism rings of supersingular elliptic curves has
emerged as a central problem for isogeny-based cryptography. The first
cryptographic application
of
isogenies between supersingular elliptic curves was
the hash function in~\cite{CGL2009}.  An
efficient algorithm for computing the endomorphism ring of a
supersingular elliptic curve would, under certain
assumptions, completely break this hash function  and also
SIKE~\cite{JDf11,SIKE}. It would also have a major impact on the security
of CSIDH~\cite{CSIDH}.

Computing the endomorphism ring of a supersingular elliptic curve $E$
was first studied by Kohel~\cite[Theorem 75]{Koh96}, who gave an
approach for generating a subring of finite index of the endomorphism
ring $\End(E)$. The algorithm was based on finding cycles in the
$\ell$-isogeny graph of supersingular elliptic curves in
characteristic $p$, and the running
time of the probabilistic algorithm was $O(p^{1 +
  \varepsilon})$. In~\cite{GPS} it is argued that heuristically one
expects $O(\log p)$ calls to a cycle finding algorithm until the
cycles generate $\End(E)$. An algorithm for computing cycles with
complexity $\tilde{O}(p^{1/2})$ and polynomial storage is given by
Delfs and Galbraith~\cite{DG16}.

One can also compute $\End(E)$ using an isogeny $\phi: \tilde{E}\to E$, where
$\tilde{E}$ is an elliptic curve with known endomorphism ring. McMurdy was
the first to compute $\End(E)$ via such an approach~\cite{M2014}, but
did not determine its complexity. In~\cite{GPS} a polynomial-time
reduction from computing $\End(E)$ to finding
an isogeny $\phi$ of powersmooth degree was given assuming some
heuristics, while~\cite{dFMPS} used an isogeny $\phi$ of $\ell$-power
degree.

In this paper we give a new algorithm for computing the endomorphism
ring of a supersingular elliptic curve $E$: first we compute two
cycles through $E$ in the supersingular $\ell$-isogeny graph that
generate an order $\Lambda$ in $\End(E)$. We show that this order
will be a Bass order with constant probability, assuming that the discriminants of the two cycles are
random in a certain way. Then we compute all maximal orders
that contain the Bass order $\Lambda$ by first solving
the problem locally, showing how to efficiently compute all
maximal superorders of $\Lambda$ when $\Lambda$ is local and
Bass. This extends work of Voight~\cite[Theorem
7.14]{V2013}. The main property of local Bass orders used here is that
there are at most $e+1$ maximal orders containing a local Bass order
$\Lambda \otimes \ZZ_q$, where $e=v_q(\discrd(\Lambda))$ is the
valuation of the reduced discriminant of $\Lambda$ (see~\cite{Brz83}).
To solve the global case, we use the local data and a local-global principle
for quaternionic orders. To bound the running time in this step, we prove that
the number of maximal global orders containing $\Lambda$ is $O(p^{\epsilon})$ for any
$\epsilon > 0$ when the size of $\Lambda$ is polynomial in $\log p$
and $\discrd(\Lambda)$ is square-free.
We conjecture that this bound also holds when $\discrd(\Lambda)$ is
not square-free.  Finally, as we compute each global maximal order, we
check if it is isomorphic to $\End(E)$.
As part of the analysis of the cycle finding algorithm,
we give a lower bound on the size of the set of all $j$-invariants $j$
that are adjacent to $j^p$ in $G(p,\ell)$, answering the lower-bound
part of Question~3 in~\cite{ACLLNSS}.

Both the algorithm for generating the suborder of $\End(E)$ and the
algorithm for computing the maximal orders containing a given order
are new.  However, our overall algorithm is still exponential: the two
cycles are found in time $O((\log p)^2p^{1/2})$, and the overall
algorithm has the same running time, assuming several heuristics.
This saves at least a factor of $\log p$ versus the previous approach
in~\cite{GPS} that finds cycles in $G(p,\ell)$ until they generate all
of $\End(E)$. This is because with that approach one expects to
compute $O(\log p)$ cycles, while our algorithm for the endomorphism
ring computes just one pair of cycles and succeeds with constant probability, assuming that the above heuristic about the discriminants of
cycles holds.  Also, our cycle finding algorithm requires only
polynomial storage, while achieving the same run time as a generic
collision-finding algorithm, which requires exponential storage.

    In the last section of the paper we give a new polynomial-time
    reduction from computing $\End(E)$ to path\-finding in the
    $\ell$-isogeny graph which is simpler in several ways than
    previous ones. For this we need to assume GRH and the heuristics of~\cite{GPS}.
We observe that this
reduction, together with the algorithms in~\cite{KLPT,DG16,GPS,EHLMP}
gives an algorithm for pathfinding in $G(p,\ell)$ that runs in time
$O((\log p)^2p^{1/2})$ and requires polynomial storage, assuming the
heuristics needed in these algorithms hold.


The paper is organized as follows. Section~\ref{sec:background}
gives some necessary background.  In Section~\ref{S: cycles} we
give an algorithm for computing cycles in the $\ell$-isogeny graph
$G(p, \ell)$ so that the corresponding endomorphisms generate an order
in the endomorphism ring of the associated elliptic curve. In
Section~\ref{section: local superorders} we show how to compute all
maximal local orders containing a given $\ZZ_q$-order $\Lambda$. In
Section~\ref{section: EndE from Lambda} we construct global orders
from these local orders and compute $\End(E)$.  In
Section~\ref{section: reduction} we give a reduction from the
endomorphism ring problem to the problem of computing $\ell$-power
isogenies in $G(p,\ell)$ that is then used to attack the second
preimage resistance of the hash function in~\cite{CGL2009}.

\section{Background on Elliptic Curves and Quaternion Algebras}\label{sec:background}
For the definition of an elliptic curve, its $j$-invariant, isogenies
of elliptic curves, their degrees, and the
dual isogeny see \cite{AEC}.

\subsection{Endomorphism rings, supersingular curves, $\ell$-power
  isogenies}
Let $E$ be an elliptic curve defined over a finite field $\FF_q$. An isogeny of $E$  to itself is called an
\textit{endomorphism} of $E$.
The set of endomorphisms of $E$ defined over $\overline{\mathbb{F}}_q$
together with the zero map is called the endomorphism ring of $E$, and
is denoted by $\End(E)$.

If the endomorphism ring of $E$ is non-commutative, $E$ is called
a {\em supersingular elliptic curve}. Otherwise we call $E$ {\em
  ordinary}. Every supersingular elliptic curve over a field of
characteristic $p$ has a model that is defined over $\mathbb{F}_{p^2}$
because the $j$-invariant of such a curve is in $\mathbb{F}_{p^2}$.

Let $E,E'$ be two supersingular elliptic curves defined over
$\FF_{p^2}$. For each prime $\ell \neq p$, $E$ and
$E'$ are connected by a chain of isogenies of degree
$\ell$. 
By~\cite[Theorem 79]{Koh96}, $E$ and $E'$ can be
connected by $m$ isogenies of degree $\ell$, where $m=O(\log p)$.  
For $\ell$ a prime different from $p$, the
        \textit{supersingular $\ell$-isogeny graph in characteristic $p$} is
        the multi-graph $G(p,\ell)$ whose vertex set is
        \vspace{-.2em}
        \[ V=V(G(p,\ell)) = \{ j \in \mathbb{F}_{p^2}: j=j(E)
  \text{
    for $E$ supersingular}\},\]
 and the number of directed edges from $j$ to $j'$ is equal to the
multiplicity of $j'$ as a root of $\Phi_{\ell}(j,Y)$. Here, given a prime $\ell$, $\Phi_{\ell}(X,Y)\in
\mathbb{Z}[X,Y]$ is the {\em modular polynomial}. This polynomial
has the
property that $\Phi_{\ell}(j,j') =0$ for $j,j' \in \mathbb{F}_q$ and $q=p^r$ if
and only if there exist elliptic curves $E(j),E(j')$ defined over
$\mathbb{F}_q$ with $j$-invariants $j,j'$ such that there is a
separable $\ell$-isogeny from $E(j)$ to $E(j')$.

\subsection{Quaternion Algebras, orders and sizes of orders}\label{quatalg}
For $a, b \in \QQ^{\times}$, let $H(a,b)$ denote the quaternion
algebra over $\QQ$ with basis $1,i,j,ij$ such that $i^2=a$, $j^2=b$
and $ij=-ji$.  That is,
$H(a,b) = \QQ + \QQ \,i + \QQ \,j + \QQ \,i j.$ Any quaternion algebra
over $\QQ$ can be written in this form.
There is a {\em canonical involution} on $H(a,b)$ which sends an
element $\alpha~=~a_1+a_2i+a_3j+a_4ij$ to $\overline{\alpha}\coloneqq a_1-a_2i-a_3j-a_4ij$.
Define the {\em reduced trace} of an element $\alpha$ as above to
be $\Trd(\alpha) = \alpha + \overline{\alpha}= 2a_1,$ and the {\em reduced
norm} to be $\Nrd(\alpha) = \alpha \overline{\alpha}= a_1^2 - aa_2^2 -ba_3^2 + aba_4^2.$

A subset $I\subseteq H(a,b)$ is a {\em lattice} if $I$ is
finitely generated as a $\ZZ$-module and  $I\otimes
\QQ\simeq H(a,b)$. If $I\subseteq H(a,b)$ is a lattice, the {\em reduced norm of $I$}, denoted
$\Nrd(I)$, is the positive generator of the fractional $\ZZ$-ideal
generated by $\{\Nrd(\alpha): \alpha \in I\}$. An {\em order} $\OO$ of
$H(a,b)$ is a subring of $H(a,b)$ which is also a lattice,
and if $\OO$ is not properly
contained in any other order, we call it a {\em maximal order}.
We call an order $\OO \subseteq H(a,b)$ {\em $q$-maximal} if $\OO \otimes
\ZZ_q$ is a maximal order in $H(a,b) \otimes \ZZ_q$.

We define $ \OO_R(I)\coloneqq \{x\in H(a,b): Ix\subseteq I\} $ to be
the {\em right order of the lattice $I$}, and we similarly define its
\textit{left order} $\OO_L(I)$. If $\OO$ is a maximal order in
$H(a,b)$ and $I\subseteq \OO$ is a left ideal of $\OO$, then
$\OO_R(I)$ is also a maximal order. Here a {\em left ideal of $\OO$} is an
additive subgroup of $\OO$ that is closed under scalar multiplication
on the left. In our setting, a lattice or an
order is always specified by a basis. The {\em size} of a lattice or
an order $\Lambda$ specified by a basis $\mathcal{B}$ in a quaternion
algebra $B$ is the number of bits needed to write down the
coefficients of the basis $\mathcal{B}$ plus the
size of $B$, which is specified by a basis and a multiplication
table. In the following we write $size(\Lambda)$ for simplicity even
though the size depends on the basis chosen to represent $\Lambda$.
We denote by $B_{p, \infty}$ the unique quaternion algebra over
$\mathbb{Q}$ that is ramified exactly at $p$ and $\infty$, and this
algebra has a standard basis~\cite[Prop.\ 5.1]{Piz80}. The
endomorphism ring of a supersingular elliptic curve is isomorphic to a
maximal order in $B_{p,\infty}$.

\subsection{Bass, Eichler, and Gorenstein orders in quaternion
  algebras; discriminants and reduced discriminants.}
\label{discriminants}
Let $B$
be a quaternion algebra over $\QQ$.  We define the {\em discriminant of
  $B$}, denoted $\disc B$, to be the product of p
rimes that ramify in $B$; then
$\disc B$ is a squarefree positive integer. If $\mathcal{O}
\subset B$ is an order, we define the {\em discriminant of
  $\mathcal{O}$} to be
$\disc(\mathcal{O}) \coloneqq  |\det(\Trd(\alpha_i \alpha_j))_{i,j}| \in \mathbb{Z}>0$,
where $\alpha_1, \dots, \alpha_4$ is a $\mathbb{Z}$-basis for
$\mathcal{O}$~\cite[\S 15.2]{Voight}.

The discriminant of an order is always a square, and the {\em reduced discriminant}
$\discrd(\mathcal{O})$ is the positive integer square root so that
$\discrd(\mathcal{O})^2 = \disc(\mathcal{O})$~\cite[\S 15.4]{Voight}.
The discriminant of an order measures how far the order is
from being a maximal order. The order $\mathcal{O}$ is
maximal if and only if $\discrd(\mathcal{O}) = \disc B$~\cite[Theorem 23.2.9]{Voight}.
Associated to a quaternion algebra $B$ over $\QQ$ there is a \textit{discriminant
form} $\Delta: B \to \QQ$, defined by $\Delta(\alpha)= \Trd(\alpha)^2 -4
\Nrd(\alpha)$, and we refer to $\Delta(\alpha)$ as the
\textit{discriminant} of $\alpha$.
Now let $\mathcal{O} \subset B$ be a $\ZZ$-order.
We say that $\mathcal{O}$ is an {\em Eichler order} if $\mathcal{O} \subseteq
B$ is the intersection of two (not necessarily distinct) maximal orders.
The \textit{codifferent} of an order is
defined as
$\codiff(\mathcal{O)} = \{ \alpha \in B: \Trd(\alpha \mathcal{O})
\subseteq \ZZ\}.$ Following \cite[Definition 24.2.1]{Voight}, we say
that $\mathcal{O}$ is {\em Gorenstein} if the lattice
codiff$(\mathcal{O})$ is invertible as a lattice as in
\cite[Definition 16.5.1]{Voight}. An order $\mathcal{O}$ is {\em Bass}
if every superorder $\mathcal{O}' \supseteq \mathcal{O}$ is
Gorenstein. An order is {\em basic} if it contains a commutative,
quadratic subalgebra $R$ such that $R$ is integrally closed in
$\QQ R$~\cite[\S 24.5]{Voight}.  Given an order $\Lambda$, its {\em radical idealizer}
$\Lambda^{\natural}$ is defined as
$\Lambda^{\natural} = \mathcal{O}_R(\operatorname{rad}\Lambda)$, where
$\operatorname{rad}\Lambda$ is the Jacobson radical of the ring $\Lambda$.  When
$B$ is a quaternion algebra over $\QQ_p$ and $\OO$ is a $\ZZ_p$-order
in $B$, we similarly define lattices, ideals, and
orders in $B$.

\section{Computing an order in the endomorphism
  ring of a supersingular elliptic curve}\label{section:
  compute end}
  \label{S: cycles}

\subsection{Computing cycles in $G(p,\ell)$}\label{section: cycles}

Fix a supersingular elliptic curve $E_0$ defined over $\FF_{p^2}$ with
$j$-invariant $j_0$.
In this section we describe and analyze an algorithm for computing two
cycles through $j_0$ in $G(p,\ell)$ that generate an order in
$\End(E_0)$.

We will first show how to construct two distinct paths from $j_0$ to
$j_0^p$.  Given two such paths $P$ and $P'$, then first traversing
through $P$ and then traversing through $P'$ in reverse gives a cycle
through $j_0$. This uses the fact that if $j$ is adjacent to $j'$,
then $j^p$ is adjacent to $(j')^p$.

Now let $P_1$ be a path of length $k$ from $j_0$ to some other vertex
$j_k$ in $G(p, \ell)$. Denote the not necessarily distinct vertices on
the path by $j_0, j_1, \dots, j_k$ and assume that $j_k$ is adjacent
to $j_k^p \in G(p,\ell)$. Let
$P^p_1 = [j_k,j_k^p,j_{k-1}^p,\dots,j_1^p,j_0^p]$.  The concatenation
$P\coloneqq P_1 P^p_1$ is a path from $j_0$ to $j_0^p$.  Such paths
were also considered in~\cite [Section 7]{CGL2009}.

If $j_0 = j_0^p$, then $P$ is already a cycle.  Otherwise, we repeat
this process to find another path $P'\coloneqq P_2 P^p_2$ that passes
through at least one vertex not in $P$.  Concatenating $P$ and $P'$
(in reverse order) gives a cycle starting and ending at $j_0$; this
corresponds to an endomorphism of $E$.  We will need the notion of a
path/cycle with no {\em backtracking} and {\em trimming a path/cycle}
to remove backtracking.

\begin{defn} Suppose $e_j, e_{j'}$ are edges in $G(p, \ell)$ that
  correspond to $\ell$-isogenies $\phi_j: E(j) \to E(j')$ and
  $\phi_{j'}: E(j') \to E(j)$ between curves $E(j)$ and $E(j')$ with
  $j$-invariants $j,j'$. We say that $e_j$ is {\em dual} to $ e_{j'}$
  if up to isomorphism $\phi_{j'}$ equals the dual isogeny
  $\hat \phi_j$ of $\phi$. That is $\phi_{j'} = \alpha \hat{\phi}_j$,
  where $\alpha \in \Aut(E(j))$. We say that a path or cycle with a
  specified start vertex $j_0$, following edges $(e_1, \dots, e_k)$
  and ending at vertex $j_k$ {\em has no backtracking} if $e_{i+1}$ is
  not dual to $e_i$ for $i =1, \dots, k-1$.
\end{defn}

In our definition, a cycle has a specified start vertex
$j_0$. According to our definition, if the first edge $e_1$ and the
last edge $e_k$ in such a cycle are dual to each other, it is not
considered backtracking.

\begin{defn}\label{def: trimming}
  Given a path $(e_1,\ldots, e_k)$ from $j_0$ to $j_k$ (with $j_0\neq
  j_k$) or a cycle with
  specified start vertex $j_0 = j_k$, define {\em trimming} as the
  process of iteratively removing pairs of adjacent dual edges until
  none are left.
\end{defn}

One can show that given a path $P$ from $j_0$ to $j_k$ with
$j_0\neq j_k$, or a cycle $C$ with start vertex $j_0 = j_k$, the
trimmed versions $\tilde{P}$ or $\tilde{C}$ may result in a smaller
set of vertices.  The vertices $j_0$ and $j_k$ will still be there in
$\tilde P$, and the only way that $j_0$ and $j_k$ may disappear from
$\tilde C$ is if the whole cycle gets removed.

\begin{defn}
  Given a path $P$ in $G_{p,\ell}$ from $j_0$ to $j_k$, we define
  $P^R$ to be the path $P$ traversed in reverse order, from $j_k$ to
  $j_0$, using the dual isogenies.
\end{defn}

  Let $S^p  \coloneqq \{j \in \FF_{p^2}\colon j \text{ is
    supersingular and $j$ is adjacent to $j^p$ in  $G(p,\ell)$}\}$.

  We can now give the algorithm to find cycle pairs:

\begin{alg}\label{alg: find end} Finding cycle pairs for prime $\ell$\\
  \noindent Input: prime $p \neq \ell$
   and a supersingular
  $j$-invariant $j_0\in\FF_{p^2}$.\\
  \noindent Output: two cycles in $G(p, \ell)$ through $j_0$.
  \begin{enumerate}
  \item \label{alg:walk} Perform $N=\Theta(\sqrt{p}\log p\log \log p)$ random walks of
    length\\ $k=\Theta(\log(p^{3/4}(\log\log p)^{1/2}))$ starting at $j_0$ and select a
    walk that hits a vertex $j_k \in S^p$, i.e.\ such that $j_k$ is $\ell$-isogenous
    to $j_k^p$; let
    $P_1$ denote the path from $j_0$ to $j_k$.
  \item Let $P_1^p$ be the path given by
    $j_k,j_k^p, j_{k-1}^p, \ldots, j_0^p$.
\item Let $P$ denote the path from $j_0$ to $j_0^p$ given as the concatenation of
  $P_1$ and $P^p_1$.  Remove any self-dual self-loops and trim $P_1P_1^p$.
\item If $j_0\in \FF_p$ then $P_1P_1^p$ is a cycle through
  $j_0$.
\item If $j_0 \in \FF_{p^2}-\FF_p$ repeat steps (\ref{alg:walk})-(3)
  again to find another path $P' =P_2P_2^p$ from $j_0$ to $j_0^p$, then
  $P(P')^R$ is a cycle. Remove any self-dual self-loops and trim the cycle.
\item Repeat Steps (1)-(5) a second time to get a second cycle.
\end{enumerate}
\end{alg}

\begin{rmk}\label{variant}
  Instead of searching for a vertex $j$ in Step (1) such that $j$ is
  adjacent to $j^p$, one could also search for a vertex $j \in
  \mathbb{F}_p$, i.e.\ $j=j^p$ or a vertex $j$ whose distance from
  $j^p$ in the graph is bounded by some fixed integer $B$. Our
  algorithm that searches for a vertex such that $j$ is adjacent to
  $j^p$ was easier to analyze because there were fewer cases to
  consider in which the trimmed cycles would not generate an order.
  \end{rmk}

To analyze the running time of
Algorithm~\ref{alg: find end}, we will use the mixing properties in
the Ramanujan graph $G(p,\ell)$.  This is captured in the following
proposition, which is an extension of~\cite[Lemma 2.1]{JMV09}
in the case that $G(p,\ell)$ is not regular or undirected (that is, when
$p\not\equiv 1 \pmod{12})$.

\begin{prop}\label{mixing}
  Let $p>3$ be prime, and let $\ell\not= p$ also be a prime. Let $S$ be any
  subset of the vertices of $G(p,\ell)$ not containing $0$ or $1728$.
  Then a random walk of length at least
  \[
  t =
  \frac{\log\left(\frac{p}{6|S|^{1/2}}\right)}{\log\left(\frac{\ell + 1}{2\sqrt{\ell}}\right)}
  \]
  will land in $S$ with probability at least $\frac{6|S|}{p}$.
  \end{prop}

One can prove this since the eigenvalues for the adjacency matrix of $G(p,\ell)$ satisfy the Ramanujan bound. This allows us to prove
 the following theorem.

\begin{thm}\label{thm: compute order}
  Let $\ell,p$ be primes such that $\ell<p/4$. Under GRH,
  Algorithm~\ref{alg: find end} computes two cycles in $G(p, \ell)$
  through $j_0$ that generate an order in the endomorphism ring of
  $E_0$ in time $O(\sqrt{p} \; (\log p)^2)$, as long as the two cycles
  do not pass through the vertices 0 or 1728, with probability
  $1-O(1/p)$. The algorithm requires $\textrm{poly} \log p$ space.
\end{thm}

\begin{rmk}\label{remark: 0 or 1728 edge case}
  In Section~\ref{section: EndE from Lambda} we use this proposition
  to compute endomorphism rings, and from this point there is no
  problem with excluding paths through 0 or 1728. This is because the
  endomorphism rings of the curves with $j$-invariants 0 and 1728 are
  known, and a path of length $\log P$, starting at $j_0$ going
  through 0 or 1728 lets us compute $\End(E_0)$ via the reduction in
  Section~\ref{section: reduction}.
\end{rmk}

\begin{proof}
  We implement Step~(\ref{alg:walk}) by letting $j_{i+1}$
  be a random root of $\Phi_{\ell}(j_i, Y)$. To test if $j \in S^p$ we
  check if $\Phi_{\ell}(j,j^p) =0$.
  Assuming GRH,  Theorem~\ref{lowerbound} below implies that,
  $|S^p| = \Omega(\sqrt{p}/\log {\log p})$ (treating $\ell$ as a
  constant). Proposition~\ref{mixing} implies that the endpoint $j_k$
  of a random path found in Step~(\ref{alg:walk}) is in $S^p$ with
  probability $\Omega(1/(\sqrt{p} \log { \log p}))$.  The probability
  that none of the $N+1$ paths land in $S^p$ is at most
  $(1-C/(\sqrt{p}\log \log p))^{N+1} \leq (1+C/(\sqrt{p}\log \log
  p))^{-(N+1)} \leq e^{-c\log p/C} = O(1/p)$ if $c=C$, where $C$ is
  from Theorem~\ref{lowerbound} and $c$ is the constant used in the
  choice of $N$.

   Now we must show that with high probability the two cycles $C_0,C_1$
  returned by the algorithm are linearly independent. We will use Corollary 4.12
  of~\cite{BCEMP}. This corollary states that two cycles $C_0$ and
  $C_1$ with no self-loops generate an order inside $\End(E_0)$ if (1)
  they do not go through 0 or 1728, (2) have no backtracking, and (3)
  have the property that one cycle contains a vertex that the other
  does not contain.

  By construction, the cycles $C_0$ and $C_1$ returned by our
  algorithm do not have any self-loops or backtracking.  To prove that
  condition (3) holds, we first claim that with high probability, the
  end vertex $j_k \in S^p$ in the path $P_1$ from $j_0$ to $j_k$ will
  not get removed when the path $P_1P_1^p$ is trimmed in Step
  (3). Then we show it's also still there in the trimmed cycle after
  Step (5).  Observe that if the path $P_1$ were to be trimmed to
  obtain a path $\tilde{P}_1$ with no backtracking, then $\tilde{P}_1$
  is still a nontrivial path that starts at $j_0$ and ends at $j_k$ as
  long as $j_0$ and $j_k$ are different which occurs with probability
  $1-O(1/p)$. After concatenating $\tilde{P}_1$ with its corresponding
  path $\tilde{P}_1^p$, the path $\tilde{P}_1\tilde{P}_1^p$ has
  backtracking only if the last edge of $\tilde{P}_1$ is dual to the
  first edge in $\tilde{P}_1^p$, i.e.\ if $j_{k-1} = j_k^p$. If that
  is the case, remove the last edge from $\tilde{P}_1$ and the first
  edge from $\tilde{P}_1^p$, and call the remaining path
  $\hat{P}_1$. The new path $\hat{P}_1$ still has the property that it
  ends in a vertex $j=j_k^p$ that is $\ell$-isogenous to its conjugate
  $(j_k^p)^p=j_k$. After concatenating $\hat{P}_1$ with its
  corresponding $\hat{P}_1^p$, this still gives a path from $j_0$ to
  $j_0^p$. Again, the concatenation of these two paths has no
  backtracking unless the last edge in $\hat{P}_1$ is the first edge
  in $\hat{P}_1^p$, i.e.\ if the last edge in $\hat{P}_1$ is an edge
  from $j_k$ to $j_k^p$. But this cannot happen, because otherwise the
  trimmed path $\tilde{P}_1$ would have backtracking because it would
  go from $j_k$ to $j_k^p$ and back to $j_k$, contradicting the
  definition of a trimmed cycle.  (With negligible probability, the
  vertex $j_k$ has multiple edges, so we exclude this case here.)
  Hence the trimmed version of $P_1P_1^p$ is $\hat{P}_1 \hat{P}_1^p$,
  and this path still contains the vertex $j_k$, since $\hat{P}_1^p$
  contains the vertex $j_k$. Now we can finish the argument by
  considering two cases:

  Case 1: $j_0 \in \FF_p$. The above argument about trimming shows
  that if the vertex $j_k$ appearing in the second cycle $C_1$ is
  different from all the vertices appearing in $C_0$ and their
  conjugates, which happens with probability $1-O(\log p/p)$, then
  that vertex $j_k$ will appear in the trimmed cycle $\tilde{C}_1$,
  but not in $\tilde{C}_0$. (This is because in this case the trimmed
  path $P_1P_1^p$ is already a cycle.) Hence by \cite[Corollary 4.12]{BCEMP},
  $\tilde{C}_0$ and $\tilde{C}_1$ are linearly independent.

  Case 2: If $j_0 \in \FF_{p^2}-\FF_p$, then with probability
  $1-O(\log(p)/p)$, the endpoint $j_k$ of $P_2$ is a vertex such that it
  or its conjugate do not appear as a vertex in $P_1$. The
  concatenation of the two paths $P=P_1P_1^p$ and $P'=P_2P_2^p$ in
  reverse is a cycle $C_0$ through $j_0$. When we trim it, it is
  still a cycle through $j_0$ in which the endpoint $j_k$ from $P_2$ appears
  because that $j_k$ or its conjugate did not appear in $P_1$.
  Similarly, Algorithm~\ref{alg: find end}  finds a second cycle $C_1$ with probability
  $1-\log(p)/p$ that contains a random vertex that was not on the first
  cycle $C_0$.
   This means that by Corollary~4.12 of~\cite{BCEMP},
  $\tilde{C}_0$ and $\tilde{C}_1$ and hence $C_0$ and $C_1$ are linearly independent.

  The running time is $O(\sqrt{p} \; (\log p)^2)$ because we are considering
  $O(\sqrt{p})$ paths of length $O(\log p)$, going from one vertex
  to the next takes time polynomial in $\ell \log p$, and we are
  assuming that $\ell$ is fixed. The storage is polynomial in $\log p$
  because we only have to store the paths $P_1,P_2$ that land in
  $S^p$.
  \end{proof}

\subsection{Determining the size of $S^p$.}
Will now determine a lower bound for the size of the set
$S^p \coloneqq \{j \in \FF_{p^2}\colon j \text{ is supersingular and
  $j$ is adjacent to $j^p$ in $G(p,\ell)$}\}$.  In \cite[Section
7]{CGL2009}, an upper bound is given for $S^p$, but in order to
estimate the chance that a path lands in $S^p$ we need a lower bound
for this set.

 Let $\ell,p$ be primes such that $\ell<p/4$.
Let $\OO_K$
        be the ring of integers of $K\coloneqq \QQ(\sqrt{-\ell p})$. We use the terminology and notation in in \cite{Elk89,BBEL}.
    Let $\Emb_{\OO_K}(\FF_{p^2})$ be the collection of pairs
        $(E,f)$ such that $E$ is an elliptic curve over $\FF_{p^2}$
        and $f\colon\OO_K \hookrightarrow \End(E)$ is a normalized
        embedding, taken up to isomorphism.
        We say $f\colon\OO_K \hookrightarrow \End({E})$ is {\em
          normalized} if each $\alpha \in \OO_K$ induces
        multiplication by its image in $\FF_{p^2}$ on the tangent space of
        ${E}$, and $(E,f)$ is isomorphic to $(E',f')$ if there exists
        an isomorphism $g: E\to E'$ such that $f(\alpha)' =
        gf(\alpha)g^{-1}$ for all $\alpha\in \OO_K$.

\begin{thm}\label{lowerbound}
Let $\ell$ be a prime and assume that $\ell<p/4$.  Let
\[
S^p= \{ j\in \FF_{p^2}: j \text{ is supersingular and } \Phi_{\ell}(j,j^p) = 0\}.
\]
Under GRH there is a constant $C>0$ (depending on $\ell$) such that
$|S^p| > C\frac{\sqrt{p}}{\log \log ( p)}$.
\end{thm}

\begin{proof}

  First, if $E$ is a
supersingular elliptic curve defined over $\FF_{p^2}$ with
$j$-invariant $j$ and $E^{(p)}$ is a curve with $j$-invariant $j^p$
and $\ell < p/4$ is also a prime, then $E$ is $\ell$-isogenous to
$E^{(p)}$ if and only if $\ZZ[\sqrt{-\ell p}]$ embeds into $\End(E)$~\cite[Lemma 6]{CGL2009}.

For any element
  $(E,f)\in \Emb_{\OO_K}(\FF_{p^2})$, $E$ is supersingular, since
  $p$ ramifies in $\QQ(\sqrt{-\ell p})$. Moreover $j(E) \in
  S^p$ by the above fact. Thus the map $\rho\colon
  \Emb_{\OO_K}(\FF_{p^2}) \to S^p$ that sends $(E,f)$ to $\rho(E,f)=j(E)$ is well-defined.

  To get a lower bound for $S^p$ we will show that for $j \in S^p$, the
  size of $\rho^{-1}(j)$ is bounded by $(\ell+1) \cdot 6$ and that
  $|\Emb_{\OO_K}(\FF_{p^2})|\gg \frac{\sqrt{\ell p}}{\log \log (\ell
    p)}$ . These two facts imply 
  \[|S^p| \geq |\Emb_{\OO_K}(\FF_{p^2})|/((\ell +1) \cdot 6) >
    \frac{1}{(\ell +1) \cdot 6} \cdot \frac{\sqrt{\ell p}}{\log \log (\ell p)}.\]

  To get a lower bound for $|\Emb_{\OO_K}(\FF_{p^2})|$ we can use
 \cite[Proposition 2.7]{GZ85} to show that $\Emb_{\OO_K}(\FF_{p^2})$
  is in bijection with $\Ell_{\OO_K}(\hat L_{\Pp})$, where
  $\hat L_{\Pp}$ is the algebraic closure of the completion of the
  ring class field $H_{\OO_K}$ at a prime $\Pp$ above $p$, and
  $\Ell_{\OO_K}(\hat L_{\Pp})$ is the set of isomorphism classes of
  elliptic curves over $\hat L_{\Pp}$ with endomorphism ring $\OO_K$.
Hence $|\Emb_{\OO_K}(\FF_{p^2})| =  |\Ell_{\OO_K}(\hat L_{\Pp})|$ whose
  order equals   $  |\Cl(\OO_K)|$. Class
  group estimates from \cite{Lit28} give
  $ |\Cl(\OO_K)|= h(-\ell p)\gg \frac{\sqrt{\ell p}}{\log \log (\ell
    p)}$.

  It remains to bound the size of $\rho^{-1}(j)$. We claim that an
  equivalence class of pairs $(E,f)$ determines an edge in
  $G(p,\ell)$. Let $[(E,f)]\in \Emb_{\OO_K}(\FF_{p^2})$ be given by
  some representative curve $E$. First assume that
  $j(E)\not= 0, 1728$. Then $(E,f)\simeq(E,g)$ implies that $f = g$,
  since $\Aut(E)=\pm1$. Thus we may identify $[(E,f)]$ with the edge
  in $G(p,\ell)$ corresponding to the kernel of $f(\sqrt{-\ell
    p})$. When $j(E)=0$ or $1728$, we may assume that $E$ is defined
  over $\FF_p$. Then let $[(E,f)]\in \Emb_{\OO_K}(\FF_{p^2})$ and
  suppose $(E,f)$ is equivalent to $(E,g)$. We can factor
  $f(\sqrt{-\ell p}) = \pi\circ \phi$ and
  $g(\sqrt{-\ell p}) = \pi \circ \phi'$, where $\phi,\phi'$ are degree
  $\ell$ endomorphisms of $E$ and $\pi$ is the Frobenius endomorphism
  of $E$. Additionally, $\pi\phi = u\pi\phi'u^{-1}$. We claim that $u$
  and $\phi$ commute. If not, then they generate an order $\Lambda$
  such that the following formula holds (see~\cite{LV15a}):
\begin{align}\label{eqn: discrd bound}
    \discrd(\Lambda) = \frac{1}{4}(\Delta(u)\Delta(\phi) - (\Trd(u)\Trd(\phi) - 2\Trd(u\hat{\phi}))^2) \leq \frac{1}{4}\Delta(u)\Delta(\phi).
\end{align}

One can show that this contradicts our assumption that $p/4>\ell$. Thus $u$ and $\phi$ commute,
and we see that $f(\sqrt{-\ell p})$ and $g(\sqrt{-\ell p})$ have the same
kernel and thus determine the same edge in $G(p,\ell)$.

We now count how many elements of $\Emb_{\OO_K}(\FF_{p^2})$ determine the same edge in $G(p,\ell)$. Suppose that $[(E,f)],[(E,g)]\in \Emb_{\OO_K}(\FF_{p^2})$ and that $\ker(f(\sqrt{-\ell p})) = \ker(g(\sqrt{-\ell p}))$. Writing $f(\sqrt{-\ell p}) = \phi \circ\pi$ and $g(\sqrt{-\ell p})=\phi' \circ\pi$ we see that $\phi$ and $\phi'$ must have the same kernel. Thus $\phi' = u\phi$ for some $u\in \Aut(E)$. Because $p>4\ell > 3$, $\Aut(E)\leq 6$ and we conclude that there are at most $6$  classes $[(E,f)]$ determining the same edge emanating from $j(E)$ in $G(p,\ell)$. Thus $|\rho^{-1}(j)|\leq (\ell + 1)\cdot 6$.
\end{proof}

Assuming GRH, this result settles the lower-bound portion of Question
3 in~\cite{ACLLNSS}. See Lemma 6 of~\cite{CGL2009} for the
upper-bound.

\section{Enumerating maximal superorders: the local
  case}\label{section: local superorders}
Let $q$ be a prime.
In this section, we give an algorithm for the following problem:

\begin{prob}
  Given a $\ZZ_q$-order $\Lambda\subseteq M_2(\QQ_q)$, find all
  maximal orders containing $\Lambda$.
\end{prob}

For general $\Lambda$ there might be an exponential number of
maximal orders containing it, so the algorithm for enumerating them
would also be exponential time. However, we will show that the above
problem can be solved efficiently when
$\Lambda$ is Bass. The main property of local Bass orders $\Lambda$ we use
is that there are at most $e+1$ maximal orders containing $\Lambda$,
where $e=v_q(\discrd(\Lambda))$~\cite[Corollary 2.5, Proposition
3.1, Corollary 3.2, and Corollary 4.3]{Brz83}.

We use the Bruhat-Tits tree $\mathcal{T}$~\cite[\S 23.5]{Voight} to
compute the maximal superorders of $\Lambda$. The vertices of
$\mathcal{T}$ are in bijection with maximal orders in $M_2(\QQ_q)$.

A homothety class of lattices $[L]\subseteq \QQ_q^2$ corresponds to a
maximal order via \begin{equation}\label{equivalence}L
  \mapsto \End_{\ZZ_q}(L)= \{x \in M_2(\QQ_q): xL \subseteq L\}
  \subseteq M_2(\QQ_q)\end{equation} for every choice of $L \in
[L]$. Two maximal orders $\OO$ and $\OO'$ are adjacent in
$\mathcal{T}$ if there exist lattices $L, L'$ for $\OO$ and $\OO'$
such that $qL\subsetneq L' \subsetneq L$. Hence the neighbors of $\OO$
in $\mathcal{T}$ correspond to the one-dimensional subspaces of
$L/qL \cong \FF_q\times \FF_q$.

A division quaternion algebra $B$ over $\QQ_q$ has only one maximal
order, which can be found using the algorithm in~\cite{V2013}. The
split case is solved by the algorithm below, and also relies on the
algorithm in~\cite{V2013}.

\begin{alg}\label{alg: enumerate local}
    Enumerate all maximal orders containing a local order \\
    Input: A $\ZZ_q$-order $\Lambda \subseteq M_2(\QQ_q)$. \\
    Output: All maximal orders in $M_2(\QQ_q)$ containing $\Lambda$.
    \begin{enumerate}
    \item\label{step: one maximal} Compute a maximal order
      $\tilde \OO\supseteq \Lambda$ with
      \cite[Algorithm~7.10]{V2013} and a lattice $\tilde L$ in
      $\QQ_q\times \QQ_q$ such that
      $\tilde O = \End_{\ZZ_q}(\tilde L)$.
        \item Let $A= \{\tilde \OO\}$ and $B=\{\tilde L\}$.
        \item While $B \not= \emptyset$:
        \begin{enumerate}
            \item Remove $L$ from $B$, and label it as discovered. Set $\OO=\End_{\ZZ_q}(L)$.
            \item Compute the set of neighbors $\mathcal{N}_{\OO}$ of $\OO$
              that contain $\Lambda$.
            \item \label{step: check containment} For each
              $\OO' \in \mathcal{N}_{\OO}$ not labeled as
              discovered, add $\OO'$ to $A$ and
              its corresponding lattice to $B$.
         \end{enumerate}
        \item Return $A$.
    \end{enumerate}
\end{alg}

Now we show that Algorithm~\ref{alg: enumerate local} is efficient
when the input lattice $\Lambda$ is Bass.

\begin{prop}\label{prop: enumerate local}
  Let $\Lambda\subseteq M_2(\QQ_q)$ be a Bass $\ZZ_q$-order, and
  $e\coloneqq v_q(\discrd(\Lambda))$. Algorithm~\ref{alg: enumerate local} computes $A \coloneqq \{\OO \supseteq \Lambda: \OO \text{ is maximal} \}$, and
  $|A|\leq e+1$. The runtime is polynomial in
  $\log q \cdot \text{size}(\Lambda)$.
\end{prop}

\begin{proof}
  To prove correctness we first show that the maximal orders
  containing an arbitrary order $\Lambda'$ in $M_2(\QQ_q)$ form a
  subtree of $\mathcal{T}$. If $\OO,\OO'$ are two maximal orders
  containing $\Lambda'$, then the maximal orders containing $\OO\cap\OO'$ are precisely the vertices in the path between
  $\OO$ and $\OO'$ in $\mathcal{T}$~\cite[\S 23.5.15]{Voight}. Each
  order on this path also contains $\Lambda'$, so the maximal orders
  containing $\Lambda'$ form a connected
  subset of $\mathcal{T}$. The above algorithm explores this
  subtree.

  If $\Lambda$ is Bass and Eichler, i.e.\ $\Lambda = \OO \cap \OO'$
  for maximal orders $\OO,\OO'$, then there are $e+1$
maximal orders containing $\Lambda$~\cite[Corollary 2.5]{Brz83}, and
they are exactly the vertices on the path from $\OO$ to $\OO'$. If
$\Lambda$ is Bass but not Eichler, then there are either 1 or 2
maximal orders containing $\Lambda$ by \cite[Proposition 3.1, Corollary
3.2, and Corollary 4.3]{Brz83}. Since they form a tree, they must also
form a path. In either case, $|A|\leq e+1$, and the vertices in $A$
form a path.

As for the running time, in Step~\ref{step: one maximal} we run
\cite[Algorithm~7.10]{V2013}, which is polynomial in
$\log q \cdot \text{size}(\Lambda)$.  Let $L$ be a lattice such that
$\OO=\End_{\ZZ_q}(L)$ contains $\Lambda$.  The neighbors of $\OO$
containing $\Lambda$ are in bijection with the lines in $L/qL$ fixed
by the action of the image of $\Lambda$ in
$\OO/q\OO\simeq M_2(\FF_q)$. For each such line, let $\bar{v}\in L/qL$
be a nonzero vector, and let $v$ be a lift to $L$. Let $w\in L$
be such that $\{v,w\}$ is a $\ZZ_q$-basis of
$L$. Then $L'\coloneqq \text{span} \{v,qw\}$ is a $\ZZ_q$-lattice such that
$\OO'\coloneqq
\End_{\ZZ_q}(L')$ contains $\Lambda$. So we can efficiently
compute the lattices $L'$ corresponding to the neighbors of $\OO$ which
contain $\Lambda$. Given such an $L'$, let $x\in M_2(\QQ_q)$ be the
base change matrix from $L$ to $L'$.
If $\mathcal{B}$ is a basis for $\OO$, then
$\mathcal{B}'\coloneqq x\mathcal{B}x^{-1}$ is a basis for $\OO'$.  The size of
$\mathcal{B}'$ is $c(\log q)+\text{size}(\OO)$ for some constant $c$,
so each neighbor of $\OO$ containing $\Lambda$ can be computed in time
polynomial in $\log q \cdot \text{size}(\OO)$.

Since the length of the path explored in the algorithm has length at
most $e$, where $e=v_q(\discrd(\Lambda))$ is polynomial in
$\text{size}(\Lambda)$, and the starting order $\tilde \OO$ is
polynomial in $\log q \cdot \text{size}(\Lambda)$ we obtain that the
size of any maximal order containing $\Lambda$ is polynomial in
$\text{size}(\Lambda) \cdot \log q$. Each step takes time polynomial
in $\log q \cdot \text{size}(\Lambda)$, so the whole algorithm is
polynomial in $\log q \cdot \text{size}(\Lambda)$.

\end{proof}

Later we will need to enumerate the $q$-maximal
$\ZZ$-orders containing a Bass $\ZZ$-order $\Lambda$.
The algorithm below uses Algorithm~\ref{alg: enumerate
  local} to accomplish this.

\begin{alg}\label{alg: enumerate q-maximal}
    Enumerate the $q$-maximal $\ZZ$-orders $\OO$ containing $\Lambda$ \\
    Input: A $\ZZ$-order $\Lambda$ and prime $q$ such that $\Lambda\otimes \ZZ_q$ is Bass.\\
    Output: All $\ZZ$-orders $\OO\supseteq \Lambda$ such that $\OO$ is $q$-maximal and $\OO\otimes \ZZ_{q'} = \Lambda \otimes \ZZ_{q'}$ for all primes $q\not=q'$.
    \begin{enumerate}
        \item\label{step: approx embed} Compute an embedding $f\colon \Lambda \otimes \QQ \hookrightarrow M_2(\QQ_q)$ such that $f(\Lambda)\subseteq M_2(\ZZ_q)$.
        \item Let $A$ be the output of Algorithm~\ref{alg: enumerate
            local} on input $f(\Lambda)$.
        \item Return $\{f^{-1}(\OO) + \Lambda : \OO \in A\}$.
    \end{enumerate}
\end{alg}

\begin{lem}
  Algorithm~\ref{alg: enumerate q-maximal} is correct. The run time is
  polynomial in $\log q \cdot \text{size}(\Lambda)$.
\end{lem}

\begin{proof}
  Step~\ref{step: approx embed} can be accomplished with Algorithms
  3.12, 7.9, and 7.10 in~\cite{V2013}, which run in time polynomial in
  $\log q \cdot \text{size}\Lambda$. For each maximal
  $\ZZ_q$-order $\OO \supseteq f(\Lambda)$, we then compute a
  corresponding $\ZZ$-lattice $\OO'\supseteq \Lambda$, whose generators
  are $\ZZ[q^{-1}]$-linear combinations of generators of $\Lambda$.
  The denominator of these coefficients is at most $q^{e}$ where
  $e\coloneqq v_q(\discrd(\Lambda))$. By Proposition~\ref{prop:
    enumerate local}, there are at most $e+1$ maximal orders
  containing $f(\Lambda)$ if $\Lambda\otimes \ZZ_q$ is Bass. It is
  straightforward to check that the lattice $\Lambda + \OO'$ is
  actually a $\ZZ$-order and has the desired completions. Moreover,
  these are all such orders by the local-global principle for orders,
  \cite[Theorem~9.5.1]{Voight}.
  \end{proof}

\begin{rmk}[The global case] Algorithm~\ref{alg: enumerate q-maximal}
  can be used to enumerate all maximal orders $\OO$ of a quaternion
  algebra $B$ over $\QQ$ that contain a $\ZZ$-order $\Lambda$ which is
  Bass, given $\Lambda$ and the factorization of $\discrd(\Lambda)$ as
  $\discrd(\Lambda)=\prod_{i=1}^m q_i^{e_i}$:

  We run Algorithm~\ref{alg: enumerate q-maximal} $m$ times, namely on
  $(\Lambda, q_1), \dots, (\Lambda, q_m)$. Let $\{X_1, \dots, X_m\}$
  be the output, where $X_i = \{\OO_{i1},\ldots,\OO_{in_i}\}$. The
  global orders containing $\Lambda$ are in bijection with
  $\prod_i{X_i}$, by associating to
  $(\OO_{1j_1},\ldots,\OO_{mj_m})\in \prod X_i$ the order
  $\sum_i \OO_{ij_i}$. In particular, the number of such orders is
  $\prod_i(e_i+1)$.  The correctness of this follows from the
  local-global principle for maximal orders
  \cite[Lemma~10.4.2]{Voight}.  The above results show that each order
  in the enumeration can be computed in time polynomial in the size of
  $\Lambda$.

  For an arbitrary order $\Lambda$, there might be an exponential
  number of global maximal orders containing it, for example when
  $\Lambda$ is the intersection of a set of representatives for the
  isomorphism classes of maximal orders. In this case it is not
  possible to compute the collection of these orders in polynomial
  time.  However, when $\Lambda$ is Bass, we can bound the number of
  maximal orders containing $\Lambda$, which is done in the next
  section.

\end{rmk}

\section{Computing $\End(E)$}\label{section: EndE
  from Lambda}
Now we describe our algorithm to compute the endomorphism ring of
$E$. By computing $\End(E)$ we mean computing a basis for an order
$\OO$ in $B_{p,\infty}$ that is isomorphic to $\End(E)$, and that we
can evaluate the basis at all points of $E$ via an isomorphism
$B_{p,\infty} \to \End(E) \otimes \mathbb{Q}$. First we give an
algorithm that uses Algorithm~\ref{alg: find end} to generate a Bass
suborder of $\End(E)$.  A heuristic about the distribution of
discriminants of cycles is used to show that just one call to
Algorithm~\ref{alg: find end} generates a Bass order with constant probability.  Then we give an algorithm which recovers $\End(E)$ from
a Bass suborder. The key property used here is that Bass orders
$\Lambda$ (whose basis is of size polynomial in $\log p$ and whose
discriminant is $O(p^k)$) only have $O(p^{\epsilon})$ maximal orders
containing them for any $\epsilon >0$. This is proved in
Proposition~\ref{numberofsuperorders} when the reduced discriminant is
square-free. Based on our numerical evidence, we conjecture that this
holds for general Bass orders as well.

  \subsection{Computing a Bass order.}\label{endoring}

Here is the algorithm for computing a Bass order.

  \begin{alg}\label{alg: compute bass}
Compute a Bass suborder $\Lambda\subseteq\End(E)$ \\
\noindent Input: A supersingular elliptic curve $E$. \\
\noindent Output: A Bass order $\Lambda \subseteq
\End(E)$ and the factorization of $\discrd(\Lambda)$, or false.

  \begin{enumerate}

\item\label{step: cycles} Compute two cycles in $G(p,\ell)$
through $j(E)$ using Algorithm~\ref{alg: find end}.

\item\label{step: traces} Let $\alpha,\beta$ be the endomorphisms
corresponding to the cycles from Step~\ref{step: cycles}.
Compute the Gram matrix for
$\Lambda = \langle 1,\alpha,\beta,\alpha\beta\rangle$.
\item\label{step: factor disc} Factor
$\discrd(\Lambda) = \prod_{i=1}^n q_i^{e_i}$.
\item If $\Lambda$ is Bass return $\Lambda$ and the factorization
  of $\discrd(\Lambda)$, else return false.
\end{enumerate}

\end{alg}
To analyze the algorithm we introduce a new heuristic:

\begin{heuristic}\label{heuristic} The probability that the
  discriminants of the two endomorphisms corresponding to the cycles
  produced by Algorithm~\ref{alg: find end} are coprime is at least
  $\mu$ for some constant $\mu>0$ not depending on $p$.
\end{heuristic}

This heuristic is based on our numerical experiments.  Intuitively, we
are assuming that the endomorphisms we compute with
Algorithm~\ref{alg: find end} have discriminants which are distributed
like random integers that satisfy the congruency conditions to be the
discriminant of an order in a quadratic imaginary field in which $p$
is inert and $\ell$ splits. Two random integers are coprime with
probability $6/\pi^2$. We are assuming that the discriminants of our
cycles are coprime with constant probability.

\begin{thm}\label{thm: compute bass}
  Assume GRH and Heuristic~\ref{heuristic}.  Then with probability at least $\mu$, Algorithm~\ref{alg: compute bass} computes a Bass order
  $\Lambda\subseteq\End(E)$, and the runtime is $O(\sqrt{p}(\log p)^2)$.
\end{thm}
\begin{proof}
 In Step~\ref{step: traces}, the Gram matrix for $\Lambda$, whose
  entries are the reduced traces of pairwise products of the basis
  elements, is computed. This uses a generalization of Schoof's algorithm (see Theorem
  A.6 of~\cite{BCEMP}), which runs in time polynomial in
  $\log p$ and $\log$ of the norm of $\alpha, \beta$. Since $\alpha$
  and $\beta$ arise from cycles of length at most
  $c \lceil \log p\rceil$, for some constant $c$ which is independent of $p$, the norms of $\alpha$ and $\beta$ are at
  most $p^c$. From the Gram matrix we can efficiently compute $\discrd(\Lambda).$

    To check that $\Lambda$ is Bass, it is enough to check that
    $\Lambda$ is Bass at each $q$ dividing $\discrd(\Lambda)$
    \cite[Theorem 1.2]{CSV}.  To check that $\Lambda$ is Bass at $q$
    it is enough to check that $\Lambda\otimes \ZZ_q$ and
    $(\Lambda\otimes\ZZ_q)^{\natural}$ are Gorenstein \cite[Corollary
    1.3]{CSV}. An order is Gorenstein if and only if its
    ternary quadratic form is primitive~\cite[Corollary
    24.2.10]{Voight}, and this can be checked efficiently. Thus, given
    a factorization of $\discrd(\Lambda)$, we can efficiently decide if
    $\Lambda$ is Bass.

    Finally, we compute the probability that the order returned by
    Algorithm~\ref{alg: find end} is Bass.  By~\cite [Theorem
    1.2]{CSV}, an order is Bass if and only if it is basic, and being
    basic is a local property. It follows that the order $\Lambda$ is
    Bass whenever the conductors of $\ZZ[\alpha]$ and $\ZZ[\beta]$ are
    coprime. A sufficient condition for this is that the discriminants
    of $\alpha$ and $\beta$ are coprime which will happen with
    probability at least $\mu$ by the above heuristic. This sufficient
    condition also covers the case when the cycle for $\alpha$ or
    $\beta$ goes through $0$ or $1728$ even though Theorem~\ref{thm:
      compute order} does not apply here.
\end{proof}

\subsection{Computing $\End(E)$ from a Bass order}
In this section we compute $\End(E)$ from a given Bass
suborder $\Lambda$. For this we enumerate the maximal
orders containing $\Lambda$ by taking sums of the $q$-maximal orders
returned by  Algorithm~\ref{alg: enumerate
  q-maximal}. As we enumerate the orders, we check each one to see if it is
isomorphic to $\End(E)$.

\begin{alg}\label{alg: enumerate and check}
  Compute $\End(E)$ from a Bass order\\
  \noindent Input: A Bass order $\Lambda\subseteq \End(E)$ with
  factored reduced discriminant $\prod_{i=1}^n q_i^{e_i}$.\\
  \noindent Output: A compact representation of $\End(E)$.
  \begin{enumerate}
    \item For each $i=1$ to $n$:
    \begin{enumerate}
    \item Compute all orders $\{\OO_{i,1},\ldots,\OO_{i,m_i}\}$
    which are maximal at $q_i$ and equal to $\Lambda$ at primes
    $q'\neq q_i$ by running
    Algorithm~\ref{alg: enumerate q-maximal} with input $\Lambda$ and
    prime $q_i$.
  \end{enumerate}

  \item Compute $f\colon \Lambda \otimes \QQ \to B_{p,\infty}$
    \item\label{step: rand O} For each choice of indices
      $(i_1,\ldots,i_n) \in [m_1]\times \cdots \times [m_n]$:
    \begin{enumerate}

  \item\label{step: create O} Set $\OO \coloneqq  \OO_{1,i_1} + \cdots + \OO_{n,i_n}$.
  \item\label{step: check O}  Compute $E'/\FF_{p^2}$ such that $\End(E')\simeq f(\OO)$ along with a compact representation of $\End(E')$.
    \item If $j(E') = j(E)$ or $j(E') = j(E)^p$, return $f(\OO)$ and the compact representation of $\End(E')$.
  \end{enumerate}
        \end{enumerate}

    \end{alg}

\begin{prop}\label{numberofsuperorders}
  Fix a positive integer $k$, and let $\Lambda$ be a Bass order whose
  size is polynomial in $\log p$ and whose reduced discriminant is
  square-free and of size $O(p^k)$.  Assume that the factorization of
  the reduced discriminant is given.  There are ${O}(p^{\epsilon})$
  maximal orders containing $\Lambda$ and Algorithm~\ref{alg:
    enumerate and check} terminates in time $\tilde{O}(p^{\epsilon})$
  for any $\epsilon >0$, assuming that the heuristics in~\cite{GPS,
    EHLMP} hold.
\end{prop}

\begin{proof}
  Computing the isomorphism
  $f\colon \Lambda\otimes \QQ \simeq B_{p,\infty}$ requires one call
  to an algorithm for factoring integers (and poly $\log p$ calls to
  algorithms for factoring polynomials over $\FF_p$, see
  ~\cite{IRS12}).  Let $\discrd(\Lambda) = p\cdot \prod_{i=1}^m q_i$
  with $q_1, \dots, q_m$ distinct and different from $p$. By the
  local-global principle for maximal orders there is one maximal order
  corresponding to each collection of $q_i$-maximal orders $\{\OO_i\}$
  with $\OO_i \supseteq \Lambda\otimes \ZZ_{q_i}$. We loop through
  these orders in Step~(\ref{step: rand O}). The size of the index set
  in that loop and hence the number of distinct maximal orders
  containing $\Lambda$ is at most $2^{\omega(\discrd(\Lambda))-1}$,
  where $\omega(n)$ denotes the number of distinct prime factors of an
  integer $n$. Fix $\epsilon > 0$. Since
  $\omega(n) = O(\frac{\log n}{\log\log n})$~\cite[Ch. 22,
  \S10]{HW08}, for $p$ large enough, the number of maximal orders
  $\OO\supseteq\Lambda$ is at most
  $2^{c'\frac{\log c\cdot p^k}{\log\log c \cdot p^k}}= (c \cdot p^k
  )^{\frac{c'}{\log \log c\cdot p^k}}$ for some $c,c' >0$, which is
  $O(p^{\epsilon})$.

  As we loop through the maximal orders $\OO$
    containing $\Lambda$, we check each one to see if it is isomorphic to
    $\End(E)$: after constructing such an order in \ref{step: rand
      O}(a), we compute in \ref{step: rand O}(b) a curve $E'$ whose
    endomorphism ring is isomorphic to $\OO$. This can be solved
    efficiently with the algorithms in to~\cite{GPS}: one computes a
    connecting ideal $I$ between $\OO$ and a special order ${\OO}'$
    and then applies Algorithm 2
    (see also Algorithm 12 of~\cite{EHLMP}). Then, in Step 3(c), we
    compare $j$-invariants. Checking each order takes time polynomial
    in $\log p$ (assuming the heuristics in~\cite{GPS, EHLMP}), so the
    total running time of the algorithm is $\tilde{O}(p^{\epsilon})$
    for any $\epsilon >0$.
\end{proof}

Our computational data from Section~\ref{computation} on the factorization pattern of the reduced discriminant of $\Lambda$ below suggest
that we will get the same running time when the reduced discriminant
of $\Lambda$ is not square-free. This motivates the following
conjecture:

\begin{conjecture}\label{conjecture: NLambda}
  Fix an integer $k\geq 0$ and assume that $\Lambda\subseteq\End(E)$
  is a Bass order of size polynomial in $\log p$ and with
  $\discrd(\Lambda) = O(p^k)$. Then for any $\epsilon >0$,
  the number of maximal orders containing $\Lambda$ is
  ${O}(p^{\epsilon})$.
\end{conjecture}

\begin{thm}\label{thm: compute EndE}
Assume GRH, Conjecture~\ref{conjecture: NLambda},
Heuristic~\ref{heuristic}, and the heuristics
in~\cite{GPS}.  Let $E$ be a supersingular
elliptic curve. Then the algorithm which combines Algorithm~\ref{alg: compute bass} and Algorithm~\ref{alg: enumerate and check} computes $\End(E)$
with probability at least $\mu$, in time $O((\log p)^2\sqrt{p})$.
\end{thm}

\begin{proof}
By the proof of Theorem~\ref{thm: compute bass}, the norms of the
endomorphisms $\alpha_1, \alpha_2$ computed by Algorithm~\ref{alg: find
  end} are bounded by $p^c$ for some constant $c$ independent of $p$,
so their discriminants satisfy
$|\Delta(\alpha_i)| < 4p^{c}$. Hence by Equation~\ref{eqn: discrd bound}, they generate an order $\Lambda$ whose reduced discriminant satisfies $\discrd(\Lambda) = O(p^{2c})$. This means we can apply Conjecture~\ref{conjecture: NLambda}, so the theorem follows from Theorem~\ref{thm: compute bass}.
\end{proof}

\subsection{Computational Data}\label{computation} We implemented a cycle finding algorithm in Sage along with an
algorithm for computing traces of cycles in $G(p,\ell)$. For each $p$
in Figure~1, and for $100$ iterations, we computed
a pair of cycles in $G(p,2)$. We then tested whether they generate a
Bass order by testing whether the two quadratic orders had coprime
conductors and computed the discriminant of the order that they
generate. We also computed an upper bound on the number of maximal
orders containing $\Lambda$ when $\Lambda$ was Bass: suppose
$\discrd(\Lambda) = p \prod_i q_i^{e_i}$, then there are at most
$N(\Lambda)\coloneqq \prod_i (e_i + 1)$ maximal orders containing
$\Lambda$.  We report how often the two cycles generated an order, how
many of those orders were Bass, and the average value of
$N(\Lambda)$. The cycle-finding algorithm we
implemented is the variant discussed in Remark~\ref{variant}: it searches
for $j\in \FF_p$ to construct the cycles using walks of length $\lceil
\log p \rceil$. We also did not
avoid a second cycle which may commute with the first since even
without that more than 80\% of cases were orders. We also
only computed cycles at $j\in \FF_{p^2} -\FF_p$ because this
is the case of interest as there are no obvious non-integer
endomorphisms.

\begin{figure}\label{fig: bass data}
  \centering
\begin{tabular}{|r | c | c | c | c}
  \hline
$p$ & orders & Bass orders & average $N(\Lambda)$ \\ \hline
30,011 & 90 & 75 & 122.37 \\ \hline
50,021 & 89 & 69 & 56.07 \\ \hline
70,001 & 92 & 76 & 122.21 \\ \hline  
90,001 & 80 & 67 & 322.04 \\ \hline 
100,003 & 81 & 75 & 337.59 \\ \hline
\end{tabular}
\caption{Results from computing 100 pairs of cycles in $G(p,2)$ at random $j \in \FF_{p^2}- \FF_p$.}
\end{figure}

\section{Computing $\End(E)$ via pathfinding in
  the $\ell$-isogeny graph }
\label{section: reduction}
In this section, we give a reduction from the endomorphism ring
problem to the problem of computing $\ell$-power isogenies in
$G(p,\ell)$, using ideas from~\cite{KLPT}, \cite{GPS}, and
\cite{EHLMP}.  This reduction is simpler than the one in~\cite{EHLMP},
and uses only one call to a pathfinding oracle (rather than poly
$\log p$ calls to an oracle for cycles in $G(p,\ell)$, as
in~\cite{EHLMP}). We apply this reduction in two ways, noting that it
gives an algorithm for computing the endomorphism ring, and that it
breaks second preimage resistance of the variable-length version of
the hash function in~\cite{CGL2009}.

\subsection{Reduction from computing $\End(E)$ to pathfinding in the
  $\ell$-isogeny graph.}
\label{ss: reduction}

We first define the pathfinding problem in the supersingular
$\ell$-isogeny graph $G(p,\ell)$:
\begin{prob}[$\ell$-PowerIsogeny]\label{EllIsogeny}
  Given a prime $p$, along with two supersingular elliptic curves $E$
  and $E'$ over $\FF_{p^2}$, output an isogeny from $E$ to $E'$
  represented as a chain of $\ell$-isogenies of length $k$ with $k$
  polynomial in $\log p$.
\end{prob}

Computing the endomorphism ring of a supersingular elliptic curve via
an oracle for $\ell$-PowerIsogeny proceeds as follows. On input $p$, Algorithm~3
of~\cite{EHLMP} returns a  supersingular elliptic curve $\tilde{E}$
defined over $\FF_{p^2}$ and a maximal order
$\tilde{\OO} \subseteq B_{p, \infty}$ with an explicit $\ZZ$-basis
$\{x_1, \ldots, x_4\}$. Proposition~3 of \cite{EHLMP} gives an
explicit isomorphism $g\colon \tilde{\OO}\to
\End(\tilde{E})$ with the property that we can efficiently evaluate $g(x_i)$
at points of $E_0$.  From this, the endomorphism ring of any
supersingular elliptic curve $E$ defined over $\FF_{p^2}$ can be
computed, given a path in $G(p,\ell)$ from
$\tilde{E}$ to $E$, with $\ell\not=p$ a small prime.

The following algorithm gives a polynomial time reduction from
computing endomorphism rings to the path-finding problem, which uses only one call to the pathfinding oracle. It assumes the
  heuristics of~\cite{GPS} and GRH (to compute $\tilde{E}$):

    \begin{alg}\label{alg: end from ell path} Reduction from computing $\End(E)$ to $\ell$-PowerIsogeny\\
  \noindent
  Input: Prime $p$, $E/\FF_{p^2}$ supersingular.
 \\
  \noindent Output: A maximal order $\OO \simeq\End(E)$, whose
  elements can be
  evaluated at any point of $E$, and a powersmooth isogeny $\psi_e: \tilde{E}\to E$, with $\tilde{E}$ as above.
  \begin{enumerate}
  \item Compute $\tilde{E},\tilde{\OO}$ with Algorithm 3 in \cite{EHLMP}.
      \item Run the oracle for pathfinding on $\tilde{E}, E$ to obtain an
        $\ell$-power isogeny $\phi=\phi_e\circ\dots\circ\phi_1:
        \tilde{E}\rightarrow E$ of degree $\ell^e$.
  \item Let $J_0 \coloneqq  \tilde{\OO}$, $P_0 \coloneqq
    \tilde{\OO}$, $\OO_0:= \tilde{\OO}$.
  \item for $k\coloneqq 1,\ldots, e$:
    \begin{enumerate}

    \item\label{step: kernel ideal} Compute $I_k\subseteq \OO_{k-1}$,
      the kernel ideal of $\phi_k$.
    \item\label{step: new connecting ideal} Compute
      $J_k\coloneqq J_{k-1}I_k$.
    \item\label{step: powersmooth ideal} Compute $P_k$, an ideal
      equivalent to $J_k$ of powersmooth norm.
    \item \label{step: afterpsi} Compute an isogeny
      $\psi_k:\tilde{E} \to E_k$ corresponding to $P_k$.
      \item Set $\OO_{k}\coloneqq  \OO_R(P_k)$.
      \end{enumerate}
    \item Return $\OO_R(P_{e}), \psi_e$.
    \end{enumerate}

  \end{alg}

  Proof sketch for correctness of reduction and running time: The
  kernel ideal ideal $I_k$, which is the ideal of $\OO_{k-1}$ of norm $\ell$
  corresponding to $\phi_k$, can be computed in polynomial time. This
  uses the fact that we can evaluate endomorphisms efficiently using
  Proposition~3 of \cite{EHLMP}. The ideal $J_k$ corresponds to
  $\psi_k: \tilde{E} \to E_k$. The algorithm is correct because at the
  $e$-th step we have $\OO_R(P_e)=\OO_R(J_e)=\End(E_e)=\End(E)$.

  \subsection{Using Algorithm~\ref{alg: end from ell
      path} to compute endomorphism rings and break second preimage of
      the CGL hash}\label{computebyreduction}
    Algorithm~\ref{alg: end from ell path} can be used to give an
    algorithm for computing the endomorphism ring of a supersingular
    elliptic curve $E$ by combining it with algorithms
    from~\cite{DG16, GPS, EHLMP}. This yields a $O((\log p)^2p^{1/2})$
    time algorithm with polynomial storage, assuming the relevant
    heuristics in~\cite{GPS,EHLMP}.

    We now consider the hash function in~\cite{CGL2009} constructed
    from Pizer's Ramanujan graphs $G(p,2)$. For each supersingular
    elliptic curve $\tilde E$, there is an associated hash function.
    The input to the hash function is a binary number of $k$ digits,
    and from this one computes a sequence of $k$ 2-isogenies, starting
    at $\tilde E$, whose composition maps to some other supersingular
    curve $E$. The $j$-invariant of $E$ is the output of the hash
    function. The following is an improvement over~\cite{EHLMP}, which
    gave a collision attack on the CGL hash for this specific hash
    function.

\begin{prop}
  Let $\tilde{E}$ be the elliptic curve computed in Step (1) of
  Algorithm~\ref{alg: end from ell path}. For the hash function
  associated to $\tilde{E}$, Algorithm~\ref{alg: end from ell path} gives
  a second preimage attack (and hence, also a collision attack)
  that runs in time polynomial in $\log p$.
\end{prop}

\begin{proof}
The attack works as follows:
Given a path from $\tilde{E}$ to $E$, use Algorithm~\ref{alg: end from ell
 path} above to compute $\End(E)$. Then use Algorithm~7
of~\cite{EHLMP} to compute new paths from $\tilde{E}$ to $E$. We note that
Algorithm~7 uses the main algorithm of~\cite{KLPT} to compute a
connecting ideal of $\ell$-power norm, whose output can be
randomized. Then for each such ideal, a corresponding path also hashes
to $j(E)$. The running time of these algorithms is polynomial in $\log
p$.
\end{proof}

\begin{rmk}
  When a start vertex $E'\neq \tilde{E}$ is chosen, the resulting hash
  function might still admit a second preimage attack if $E'$ was
  obtained by choosing a path of $\log p$ from $\tilde{E}$ to $E'$ so that
  the endomorphism ring of $E'$ is known.
  \end{rmk}

\section{Acknowledgements}
We would like to thank Daniel Smertnig and John Voight for several helpful
discussions and suggestions. We thank Ben Diamond for alerting us that an
algorithm similar to Algorithm \ref{alg: end from ell path} in
Section~\ref{section: reduction} already appeared in~\cite{dFMPS}. Finally we
would like to thank an anonymous reviewer of a previous version of this paper
whose suggestions greatly simplified Section~\ref{section: local superorders}.


\end{document}